\documentclass[12pt,a4paper,psamsfonts]{amsart}
\usepackage{amssymb,amscd,amsxtra,calc}
\usepackage{cmmib57}

\theoremstyle{plain}
    \newtheorem{thm}{Theorem}[section]
    
    \newtheorem{claim}[thm]{Claim}
    \newtheorem{corollary}[thm]{Corollary}
    \newtheorem{lemma}[thm]{Lemma}
    
    \newtheorem{conjecture}[thm]{Conjecture}
    
    \newtheorem{theorem}[thm]{Theorem}

\theoremstyle{definition}

    \newtheorem{remark}[thm]{Remark}
\theoremstyle{remark}

    \newtheorem{setup}[thm]{}

\newcommand{\C}{\mathbb{C}}

\newcommand{\Q}{\mathbb{Q}}
\newcommand{\PP}{\mathbb{P}}

\newcommand{\R}{\mathbb{R}}

\newcommand{\Z}{\mathbb{Z}}

\newcommand{\SD}{\mathcal{D}}

\newcommand{\Exc}{\operatorname{Exc}}

\newcommand{\id}{\operatorname{id}}

\newcommand{\NE}{\overline{\operatorname{NE}}}

\newcommand{\NS}{\operatorname{NS}}

\newcommand{\Sing}{\operatorname{Sing}}
\newcommand{\Supp}{\operatorname{Supp}}

\begin{document}

\title[Positivity criteria for log canonical divisors and hyperbolicity]
{Positivity criteria for log canonical divisors and hyperbolicity}

\author{Steven S.~Y.~Lu}
\address
{
\textsc{D\'epartement de Math\'ematiques} \endgraf
\textsc{UQAM, C.P. 8888, Succersale Centreville, Montr\'eal,
Qu\'ebec, Cananda  H3C 3P8
}}
\email{lu.steven@uqam.ca}

\author{De-Qi Zhang}
\address
{
\textsc{Department of Mathematics} \endgraf
\textsc{National University of Singapore, 10 Lower Kent Ridge Road,
Singapore 119076
}}
\email{matzdq@nus.edu.sg}

\begin{abstract}
Let $X$ be a complex projective variety
and $D$ a reduced divisor on $X$.
Under mild conditions on the singularities
of $(X, D)$, which includes the case of smooth
$X$ with simple normal crossing $D$, and by running
the minimal model program,
we obtain by induction on dimension via adjunction
geometric criteria guaranteeing various positivity
conditions for $K_X+D$.
Our geometric criterion for $K_X+D$
to be numerically effective yields also
a geometric version of the cone theorem
and
a criterion for $K_X+D$ to be pseudo-effective
with mild hypothesis on $D$.
We also obtain,
assuming the abundance conjecture and
the existence of rational curves on Calabi-Yau manifolds,
an optimal geometric sharpening of the Nakai-Moishezon
criterion for the ampleness of a divisor of the form $K_X+D$,
a criterion verified under a canonical hyperbolicity assumption on
$(X,D)$. Without these conjectures, we verify
this ampleness criterion with mild assumptions on $D$, being none in
dimension two and $D\neq 0$ in dimension three.\\[-12mm]
\end{abstract}




\maketitle


\section{Introduction}

We start with some preliminaries before stating our main results.
We work over the field $\C$ of complex numbers.
An algebraic variety $X$ is called {\it Brody hyperbolic} = BH
(resp.~Mori hyperbolic = MH)
if the following hypothesis (BH) (resp.~(MH)) is satisfied:

\begin{itemize}
\item [(BH)]
Every holomorphic map from the complex line $\C$ to $X$ is a constant map.
\item[(MH)]
Every algebraic morphism from the complex line $\C$ to $X$ is a constant map.\\
(Eqivalently, no algebraic curve in $X$ has normalization equal to
$\PP^1$ or $\C$.)
\end{itemize}


Consider a pair $(X, D)$ of an (irreducible) algebraic variety $X$ and a
(not necessarily irreducible
or equi-dimensional) Zariski-closed subset $D$ of $X$.
When $X$ is an irreducible curve and $D$ a finite subset of $X$,
the pair $(X, D)$ is {\it Brody hyperbolic}
(resp.~MH)
if $X \setminus D$ is
Brody hyperbolic (resp.~MH).
Inductively, for an algebraic variety $X$
and a Zariski-closed subset $D$ of $X$,
with $D = \cup_i D_i$ the irreducible decomposition, the pair $(X, D)$ is
{\it Brody hyperbolic} (resp.~MH)
if
$X \setminus D$ and the pairs $(D_k, D_k \cap (\cup_{i \ne k} D_i))$ for all $k$, are so.
Note that the Zariski-closed subset $D_k \cap (\cup_{i \ne k} D_i)$ may not be equi-dimensional in general.\\[-5mm]

Here, it is crucial 
that $D$ is a divisor and all the $D_k$'s are
at least $\Q$-Cartier
(so as to have equidimensionality above
for example). But we will work with
a bit less from the outset as demanded by our general setup.\\[-4mm]

The pair $(X, D)$ is called {\it projective} if $X$ is a projective variety and {\it log smooth} if
further $X$ is smooth and $D$ is a reduced {\it divisor} with simple normal crossings.\\[-5mm]


Recall that a divisor $F$ on a smooth projective variety
$X$ is called {\it numerically effective}
(nef), if $F.C\geq 0$ for all curves $C$ on $X$.
We consider the following conjecture:\\[-8mm]

\begin{conjecture}\label{ConjA}
Let $(X, D)$ be a log smooth projective Mori hyperbolic
$($resp.~Brody hyperbolic$)$ pair.
Then the log canonical divisor $K_X + D$ is nef
$($resp.~ample$)$.\\[-8mm]
\end{conjecture}

This conjecture for the case of ample $K_X+D$ with $D=0$
is one of the oldest conjectures in higher dimensional
hyperbolic geometry and is currently known only up to
dimension two. See \cite{Lang}.
Our main results give
affirmative answers to Conjecture \ref{ConjA} but with
further assumptions for the ampleness of $K_X + D$,
such as the ampleness of
at least $n-3$ irreducible
Cartier components of a nonzero $D$ for
$n=\dim\,(X, D) := \dim X$.

Let $D$ be a reduced Weil divisor on a projective variety $X$
and $\mathcal D$ the collection of terms in its irreducible
decomposition $D = \sum_{i=1}^s D_i$. A stratum
of $D$ is a set of the form
$\cap_{i \in I} D_i$ ($= X$ when $I = \emptyset$)
for some partition $\{1, \dots, s\} = I \coprod J$.
We call a $\mathcal D$-{\it rational curve} a
rational curve $\ell$ in a stratum $\cap_{i \in I} D_i$ of $D$
such that the normalization of $\ell \setminus (\cup_{j \in J} D_j)$
contains the complex line $\C$. We call a $\mathcal D$-{\it algebraic $1$-torus}
a rational curve $\ell$ in a stratum $\cap_{i \in I} D_i$ of $D$ such that
$\ell \setminus (\cup_{j \in J} D_j)$ has the
$1$-dimensional algebraic torus
$\C^* = \C \setminus \{0\}$
as its normalization.
We call a closed subvariety $W\subseteq X$ a
$\mathcal{D}$-{\it compact variety}
if $W$ lies in some stratum $\cap_{i \in I} D_i$ of $D$
and $W\cap \cup_{j \in J} D_j=\emptyset$.
A $\mathcal{D}$-compact variety $W \subseteq X$ is called a
$\mathcal{D}$-compact rational curve,
a $\mathcal{D}$-{\it torus},
a $\mathcal{D}$-$\Q$-{\it torus}, or a $\mathcal{D}$-$\Q$-{\it CY}
variety
if $W$ is a rational curve, an abelian variety,
a $\Q$-torus, or a $\Q$-CY variety, respectively.

Here a projective variety is called a {\it $\Q$-torus} if
it has an abelian variety as its finite \'etale (Galois)
cover. A projective variety $X$ with only klt singularities
(see \S \ref{setup1} or \cite{KM})
is called a $\Q$-{\it CY variety} if some positive multiple of
its canonical divisor $K_X$ is linearly equivalent to
the trivial divisor. A simply connected
$\Q$-{\it CY} surface with only Du Val
singularities is called a {\it normal K3 surface}.
A $\Q$-torus is an example of a smooth $\Q$-CY.\\[-6mm]

\begin{theorem}\label{ThF'}
Let $X$ be a smooth projective variety of dimension $n$
and $D$ a reduced
divisor on $X$ with simple normal crossings.
Then the closure $\NE(X)$ of effective $1$-cycles on $X$ is generated by
the $(K_X + D)$-non negative part and
at most a countable collection of
extremal $\mathcal{D}$-rational curves
$\{\ell_i\}_{i \in N}, \, N \subseteq \Z$:
$$\NE(X) = \NE(X)_{K_X + D \ge 0} +
\sum_{i \in N} \R_{> 0}[\ell_i] .$$
Further, $-\ell_i . (K_X + D)$ is in $\{1, 2, \dots, 2n\}$,
and is equal to $1$
when $\ell_i$ is not a $\mathcal{D}$-{\rm compact} rational curve.
In particular, if the pair $(X, D)$ is Mori hyperbolic
then $K_X + D$ is nef.
\end{theorem}

Theorem \ref{ThEE''} below generalizes,
to arbitrary dimensions for pairs, of results that were only known
up to dimension two in the case $D=0$ concerning hyperbolic
varieties. This is a partial answer to the ampleness part of Conjecture \ref{ConjA}
which seems to be very difficult. See \S \ref{statement} for more general results but
where we assume two deep conjectures.

\begin{theorem}\label{ThEE''}
Let $X$ be a smooth projective variety of dimension $n$
and $D$ a reduced
divisor on $X$ with simple normal crossings such that
$D$ has at least $n-2$
irreducible components amongst which at least $n-3$ are ample.
If $K_X+D$ is not ample, then either
it has non-positive degree on a $\mathcal{D}$-rational curve
or on a $\mathcal{D}$-algebraic $1$-torus,
or some positive multiple of
it is linearly equivalent to
the trivial divisor on
a $\mathcal{D}$-torus
$T$ with $\dim T \le 2$.
In particular, if $(X,D)$ is Brody hyperbolic, then $K_X+D$ is ample.
\end{theorem}

We remark that Theorems \ref{ThF'}, and \ref{ThEE''}
are special cases of Theorems \ref{ThF}, and \ref{ThE'} or \ref{ThE} (see also Theorem \ref{ThD})
where we allow the pair
$(X,D)$ to be singular
and $D$ to be
augmented with a fractional
divisor
$\Gamma$ in a setting that is natural to our approach.
Theorem \ref{ThF'} is also obtained in
McQuillan-Pacienza \cite{MP} by a different method, but not our more
general result in the singular case, Theorem \ref{ThF}.

By running the minimal model program, an
easy consequence of the cone theorem is the following criterion for
$K_X+D$ to be peudo-effective (the weakest form of
positivity in birational geometry) under a mild condition on $D$.
It is a special case of Theorem \ref{ThH}.

\begin{theorem}\label{ThH'}
Let $X$ be a smooth projective variety
and $D$ a nonzero simple normal crossing reduced divisor
no component of which is uniruled.
Then $K_X + D$ is not pseudo-effective if and only if
$D$ has exactly one irreducible component and
$X$ is dominated by $\mathcal D$-rational curves $\ell$
with $-\ell . (K_X + D) = 1$, $\ell \cong \PP^1$ and
$\ell \setminus \ell \cap D \cong \C$ for a general $\ell$.

In particular, $K_X+D$ is pseudo-effective for a log
smooth pair $(X,D)$ such that
$D$ contains two or more non-uniruled components.
\end{theorem}

Key to our proof in the presence of Cartier boundary divisors
is Kawamata's result
on the length of extremal rays, a fundamental
result in the subject (see Lemma~\ref{length}).
Our proofs of Theorems $\ref{ThE} \sim \ref{ThD}$ are inductive in nature and
reduce the problem to questions on
adjoint divisors
in lower dimensions by adjunction.

\par \vskip 1pc \noindent
{\bf Acknowledgement.}
The authors are partially supported by an NSERC discovery grant
and an ARF of NUS.
This paper is done during the visits of the second author to UQAM
in August 2011 and July 2012, with
the support and the excellent research environment
provided by CIRGET there.
The authors thank the referees for the careful reading, and
are well indebted to answers and suggestions
from experts and colleagues.

\section{Preliminary results}

We use the notation and terminology in \cite{KM}, which is a
basic reference for the log minimal model program (LMMP).
We run the LMMP in our proofs,
but detailed techniques or knowledge of LMMP are not required
as precise references are always given.

By a divisor on a normal variety, we will always mean a Weil $\Q$-divisor.

\begin{setup}\label{setup1}
{\rm
Let $L$ be a $\Q$-Cartier divisor on a projective variety $X$,
i.e. some integer multiple of $L$ is Cartier.
$L$ is {\it pseudo-effective} if its class in the N\'eron-Severi space
$\NS_{\R}(X) := \NS(X) \otimes \R$
is in the closure of the
cone generated by the classes of effective divisors in $\NS_{\R}(X)$.
$L$ is {\it numerically effective} (or  {\it nef\,}) if $\deg(L_{| C}) \ge 0$
for every curve $C$ on $X$.
For two $\Q$-Cartier divisors $L_i$, we write $L_1 \sim_{\Q} L_2$
if $mL_1 \sim mL_2$ (linear equivalence) for some $m \ge 1$.\\[-5mm]

\noindent
Let $\Gamma = \sum_{i=1}^s a_i \Gamma_i$ be a divisor on $X$
with $\Gamma_i$ distinct irreducible divisors and $a_i \in \Q$.
Its {\it integral part} is given by
$\lfloor \Gamma \rfloor := \sum_{i=1}^s \lfloor a_i \rfloor \Gamma_i$
where $\lfloor a_i \rfloor$ is the integral part of $a_i$.\\[-5mm]

\noindent
Recall that a pair $(X, \Delta)$ of an effective
Weil $\Q$-divisor $\Delta$ on a normal variety $X$ is called
{\it divisorial log terminal = dlt}
(resp.~{\it Kawamata log terminal = klt}, or {\it log canonical = lc}) if for some
log resolution (resp.~for one (or equivalently for all) log resolution),
the discrepancy of every exceptional divisor (resp.~of every divisor) on the resolution
satisfies $> -1$ (resp.~$> -1$, or $\ge -1$).
In all cases, $K_X+\Delta$ is $\Q$-Cartier.
A dlt pair $(X, \Delta)$ is klt if and only if $\lfloor \Delta \rfloor = 0$.
In case $\dim X = 2$, if $(X, \Delta)$ is dlt, then $X$
is $\Q$-factorial (so that all Weil divisors are $\Q$-Cartier).
We refer to \cite[Definition 2.34]{KM} and \cite[\S 3]{Co} for details.
}
\end{setup}

The following adjunction formula
is key for our inductive process,
see \cite[Propositions 3.9.2 and 3.9.4]{Co} for a concise statement,
and \cite[Proposition 16.6, 16.7]{Ko} for the proof.
The pair in Lemma \ref{dlt} being dlt
implies by adjunction that
all irreducible components $D_i$ ($1 \le i \le s$) of $D$
as well as 
those of their intersections
are normal varieties, see \cite[Corollary 5.52]{KM}.
$D_1$ intersects $D - D_1$ transversally at a general point of their intersection
(see \cite[Proposition 16.6.1.1]{Ko}).
In particular,
$(D - D_1)_{| D_1}$ is a well-defined Weil divisor.

\begin{lemma}\label{dlt}
Let $X$ be a normal variety, $D$ a reduced Weil divisor and
$\Gamma$ an effective Weil $\Q$-divisor such
that the pair $(X, D + \Gamma)$ is dlt.
Let $D_1 \subseteq D$ be an irreducible component.
Then $D_1$ is a normal variety and there is an effective
Weil $\Q$-divisor $\Delta$ such that
$$(K_X + D + \Gamma)_{| D_1} = K_{D_1} + (D - D_1)_{| D_1} + \Delta ,$$
$(D - D_1)_{| D_1}$ is a reduced Weil divisor
and the pair $(D_1, (D - D_1)_{| D_1} + \Delta)$ is dlt.
\end{lemma}

\begin{remark}
For a log smooth pair $(X, D)$, it is easily seen that if 
$(X, D)$ is Brody (resp.~Mori) hyperbolic
then so are the pairs $(D_1, (D-D_1)_{| D_1})$
and $(F, D_{|F})$,
where $D_1$ is any irreducible component of $D$
and $F$ is a connected general fibre of
a fibration $X \to Y$ onto a normal variety $Y$.
By Lemma \ref{dlt}, the
same claim holds if $(X, D)$ is dlt instead of log smooth
(used in Lemma \ref{restr}).
We also note that if $(X, D)$ is dlt then so is $(F, D_{|F})$.
Indeed, take a resolution $\tau: X' \to X$ and the induced
fibration $X' \to Y$ with general fibre $F'$.
Then $K_{F'} = K_{X' | F'}$ by adjunction.
Since $F$ is normal and chosen to meet transversally
the singular locus $\Sing X$ of $X$ at the general points of
$\Sing X$ and since $\tau_*K_{F'} = K_F$, we also have $K_{X | F} = K_F$.
We can now compare the discrepancies of $(X, D)$
and $(F, D_{|F})$ to see that the latter is still dlt
(e.g.~by applying \cite[Lemma 5.17(1)]{KM}).
\end{remark}

\section{Main results for singular pairs}\label{statement}

Let $X$ be a normal projective variety and $D$ a reduced Cartier divisor on $X$.
The pair $(X, D)$ is called {\it BH, or MH with respect to a $($Cartier$)$
decomposition} of $D = \sum_{i=1}^s D_i$,
if  each $D_i$ is a reduced
(Cartier) divisor on $X$, not necessarily irreducible,
and both $X \setminus D$ and $(\cap_{i \in I} D_i) \setminus (\cup_{j \in J} D_j)$
are respectively BH, or MH
for every partition of $\{1, \dots, s\} = I \coprod J$.
Note that, for a log smooth pair $(X, D)$,
the respective definitions of hyperbolicity given in \S 1.1
are equivalent to the
definitions of hyperbolicity above with respect to the irreducible
decomposition of $D$.

Theorem \ref{ThF} below
includes Theorems \ref{ThF'} as a special case.

\begin{theorem}\label{ThF}
Let $X$ be a
projective variety of dimension $n$,
$D$ a reduced divisor that is the sum of a
collection $\mathcal D=\{D_i\}_{i=1}^s$ of reduced
Cartier divisors
for some $s \ge 0$
and $\Gamma$
an effective Weil $\Q$-divisor on $X$ such that the pair $(X, D + \Gamma)$
is divisorial log terminal.
Then the closure $\NE(X)$ of effective $1$-cycles on $X$ is generated by
the $(K_X + D + \Gamma)$-non negative part and
at most a countable collection of
extremal $\mathcal{D}$-rational curves
$\{\ell_i\}_{i \in N}, \, N \subseteq \Z$:
$$\NE(X) = \NE(X)_{K_X + D + \Gamma \ge 0} +
\sum_{i \in N} \R_{> 0}[\ell_i] .$$
Furthermore,
$-\ell_i . (K_X + D+\Gamma)$ is in $(0, 2n]$,
and is even in $(0, 2)$
if $\ell_i$ is not $\mathcal{D}$-{\rm compact}.
In~particular, if the pair $(X,D)$ is Mori hyperbolic with respect to the
Cartier decomposition $D = \sum_{i=1}^s D_i$ then $K_X+D+\Gamma$ is nef.
\end{theorem}

To state our general results on ampleness,
we will need to refer to the following two
standard conjectures on the structure of algebraic varieties.

\begin{conjecture}\label{ConjAA}{\rm\bf Abundance($l$):}
Let $(X, D)$ be an $l$-dimensional log smooth projective pair
whose log canonical divisor $K_X + D$ is nef.
Then some positive multiple of $K_X+D$ is base point free,
i.e. $K_X+D$ is semi-ample.
\end{conjecture}

\begin{conjecture}\label{ConjAAA} {\rm\bf CY($m$):}
Let $X$ be an $m$-dimensional simply connected
smooth projective variety with trivial canonical line bundle,
i.e. a Calabi-Yau manifold.
Then $X$ contains a rational curve.
\end{conjecture}

We remark that {\rm  Abundance($l$)} is known to hold
for $l \leq 3$ (even for dlt or lc pairs, see \S \ref{setup1})
and {\rm CY($m$)} is known to hold
for $m\leq 2$ (see \cite{KeMaMc}, \cite{MM}).
Since the two conjectures are known for $r = 3$,
Theorem \ref{ThEE''} is an immediate corollary of Theorem \ref{ThE'} (or \ref{ThE}) below.
\\[-6mm]

\begin{theorem}\label{ThE'}
Let $X$ be a smooth projective variety of dimension $n$
and $D$ a reduced
divisor on $X$ with simple normal crossings.
Fix an $r$ in $\{1, 2, \dots, n\}$.
Assume that {\rm Abundance($l$)} holds
for $l \leq r$  and that
$D$ has at least $n-r+1$
irreducible components amongst which at least $n-r$ are ample.
If $K_X+D$ is not ample, then either
it has non-positive degree on a $\mathcal{D}$-rational curve
or on a $\mathcal{D}$-algebraic $1$-torus,
or some positive multiple of
it is linearly equivalent to the trivial divisor on
a smooth $\mathcal{D}$-$\Q$-CY
variety $T$ of dimension $< r$.
We can take $T$ to be a $\Q$-torus if
further {\rm CY($m$)} holds for $m<r$.

In particular, if {\rm Abundance($l$)}
and {\rm CY($m$)} hold for $l \leq r$ and $m<r$
and the pair $(X,D)$ is Brody hyperbolic then $K_X+D$ is ample.
\end{theorem}

Theorems \ref{ThE} and \ref{ThH} below
include Theorems \ref{ThE'} (and \ref{ThEE''}) and \ref{ThH'} as special cases.

One may take $\Gamma = 0$ in Theorems $\ref{ThF}, \ref{ThE}$ and $\ref{ThD}$.
Such a $\Gamma$ naturally appears in our inductive procedure as the `different'
in the adjunction formula Lemma \ref{dlt}.
See Remark \ref{PropD} for the case where $D$ is not Cartier.
We remark that
the cone theorem, which we use, has been generalized by Ambro and Fujino
to log canonical pairs, see \cite{Am}, \cite{Fu2}.

Note that {\rm Abundance($l$)} (for dlt pairs)
and {\rm CY($m$)}
always hold for $l \leq 3$ and $m\leq 2$ (\cite{KeMaMc}, \cite{MM}) and that a normal K3 surface has infinitely many
elliptic curves (\cite{MM}).

\begin{theorem}\label{ThE}
Let $X$ be a
projective variety of dimension $n$,
$D$ a reduced Weil divisor and $\Gamma$
an effective Weil $\Q$-divisor on $X$ such that the pair $(X, D + \Gamma)$
is divisorial log terminal $($dlt$)$.
Assume one of the following three conditions.
\begin{itemize}
\item[(1)]
$n \le 2$ and $D$ is Cartier (cf. Remark \ref{PropD} for the case with $D$ non-Cartier).
\item[(2)]
$n = 3$, $D$ is nonzero and Cartier.
\item[(3)]
$n \ge 4$, the pair $(X, D + \lfloor \Gamma \rfloor)$ is log smooth,
there is an $r \in \{1, 2, \dots, n\}$ such that
$D$ has at least $n-r+1$
irreducible components amongst which at least $n-r$ are ample
and {\rm Abundance($l$)} $($for dlt pairs$)$
holds for $l \leq r$.
\end{itemize}
Then $K_X+D+\Gamma$ is ample, unless either
$(K_X + D + \Gamma)$ has non-positive degree on a
$\mathcal{D}$-rational curve or on a $\mathcal{D}$-algebraic $1$-torus,
or the restriction $(K_X + D + \Gamma)_{| T} \sim_{\Q} 0$
for a $\mathcal{D}$-$\Q$-CY variety $T$.
When $T$ is singular, it is
a normal $K3$ surface and $n\in \{2, 3\}$.
Otherwise, $T$ can be taken to be an abelian surface
or an elliptic curve in Cases~$(1)$,~$(2)$ and,
if further {\rm CY($m$)} holds for $m<r$,
taken to be a $\Q$-torus with $ \dim T < r$
in Case~$(3)$.

In particular, if the pair $(X, D)$ is Brody hyperbolic
then $K_X+D+\Gamma$ is ample, provided that either $n \le 3$ or
{\rm CY($m$)} holds for $m<r$.
\end{theorem}

\begin{theorem}\label{ThD}
Let $X$ be a
projective variety of dimension $n$,
$D$ a reduced divisor that is the sum of a
collection $\mathcal D=\{D_i\}_{i=1}^s$ of reduced
Cartier divisors for some $s \ge 1$ and
$\Gamma$
an effective Weil $\Q$-divisor on $X$ such that the pair $(X, D + \Gamma)$
is divisorial log terminal.
Assume one of the following two conditions.
\begin{itemize}
\item[(1)]
$n = 4$, $s \ge 2$ and $D_1$ is irreducible and ample.
\item[(2)]
$n \ge 4$, $s \ge n-r +1$ for some $r \in \{1, 2, \dots, n\}$,
$D_j$ is ample for all $j \le n-r+1$ and {\rm Abundance($l$)}
$($for dlt pairs$)$ holds for $l \leq  r$.
\end{itemize}
Then $K_X + D + \Gamma$ is ample, unless
either $(K_X + D + \Gamma)$ has non-positive degree on a
$\mathcal{D}$-rational curve or on a $\mathcal{D}$-algebraic $1$-torus, or
the restriction
$(K_X + D + \Gamma)_{| T} \sim_{\Q} 0$
for some $\mathcal{D}$-$\Q$-CY
variety $T$ $( \dim T < r$ in Case$(2))$.
$T$ is smooth when
the pair
$(X, D + \lfloor \Gamma \rfloor)$ is log smooth
and can be taken to be a $\Q$-torus when further {\rm CY($m$)} holds for $m<r$.
\end{theorem}

\begin{theorem}\label{ThH}
Let $X$ be a $\Q$-factorial projective variety,
and $D$ a nonzero reduced Weil divisor with $D = \sum_{i=1}^s D_i$ the irreducible decomposition
such that the pair $(X,D)$ is divisorial log terminal.
Assume that no component of $D$ is uniruled
and $K_X + D$ is not pseudo-effective.
Then $D$ has exactly one irreducible component
and $X$ is dominated by $\mathcal D$-rational curves $\ell$
with $-\ell . (K_X + D) = 1$, $\ell \cong \PP^1$ and $\ell \setminus \ell \cap D \cong \C$
for a general $\ell$.
\end{theorem}

\noindent
For the converse direction of Theorem \ref{ThH}, see \cite[Lemma 5.11]{KMc}.

\par \vskip 1pc \noindent
{\bf N.B.} In Theorem \ref{ThD}, $D_1$ needs to be irreducible
so that Theorem \ref{ThE} (2) can be applied:
if $D_1 = D_{11} + D_{12}$ is reducible
then the restriction $D_{2 |D_{11}}$ might be zero.
{\it The ampleness assumption on
some $D_i$'s in various theorems is due to the same reason.}

\section{Proofs of the theorems in Section~\ref{statement}}

%
%
%
%


Let $n \ge 1$.
Let $(X, D + \Gamma)$
be a pair as in Theorem {\rm \ref{ThF} (resp.~\ref{ThE} (3), or \ref{ThD} (2))}
but with $\dim X = n+1$.
Let $Y$ be an irreducible component of $D + \lfloor \Gamma \rfloor$.
By
Lemma \ref{dlt},
$$(K_X + D + \Gamma)_{| Y} = K_{Y} + (D + \lfloor \Gamma \rfloor - Y)_{| Y}  + \Delta
= K_Y + D_Y + \Gamma_Y$$
where $\Delta \ge 0$ and the pair $(Y, D_Y + \Gamma_Y)$ is dlt. Here,
if a Cartier decomposition
$D = \sum_{i=1}^s D_i$ is involved, we define
$D_Y := (\sum_{i \ne k} D_i)_{| Y}$ and
$\Gamma_Y := (D_k - Y + \lfloor \Gamma \rfloor)_{| Y} + \Delta$
when $Y \subseteq D_k$ (resp.
$D_Y := D_{| Y}$ and $\Gamma_Y := (\lfloor \Gamma \rfloor - Y)_{| Y} + \Delta$
when $Y \subseteq \lfloor \Gamma \rfloor$).
Otherwise, we define
$D_Y := (D - Y)_{| Y}$ and $\Gamma_Y := \lfloor \Gamma \rfloor_{| Y} + \Delta$
when $Y \subseteq D$ (resp.
$D_Y := D_{| Y}$ and $\Gamma_Y := (\lfloor \Gamma \rfloor - Y)_{| Y} + \Delta$
when $Y \subseteq \lfloor \Gamma \rfloor$).

One can easily verify:

\begin{lemma}\label{restr}
With the above assumption and notation,
the pair $(Y, D_Y + \Gamma_Y)$
satisfies the respective conditions in Theorems $\ref{ThF}$, $\ref{ThE} \, (3)$
and $\ref{ThD} \, (2)$, and $\dim Y = n$.
\end{lemma}

Kawamata's theorem, \cite[Theorem 1]{Ka},
on the length of extremal rays is key to us.

\begin{lemma}\label{length} {\rm (see \cite[Theorem 1]{Ka})}
Let $(X, \Delta)$ be a dlt pair and $g : X \to Y$ the contraction of
a $(K_X + \Delta)$-negative extremal
ray $R = \R_{ > 0} [\ell]$.
Let $E$ be an irreducible component of the $g$-exceptional locus $\Exc(g)$, i.e.,
the set of points at which $g$ is not locally an isomorphism.
Let $d := \dim E - \dim g(E)$. Then we have:
\begin{itemize}
\item[(1)]
$E$ is covered by a family of $($extremal$)$ rational curves
$\{\ell_t\}$ such that $g(\ell_t)$ is a point
$($and hence the class $[\ell_t] \in R )$.
In particular, $E$ is uniruled.
\item[(2)]
Suppose further that $d = 1$.
Then the $\ell_t$ in $(1)$ can be chosen such that
$- \ell_t . (K_X + \Delta) \le 2d = 2$.
\end{itemize}
\end{lemma}

\begin{proof}
If $(X, \Delta)$ is klt, then the lemma is just \cite[Theorem 1]{Ka}.
When $(X, \Delta)$ is dlt, we have by \cite[Proposition 2.43]{KM}
that for any ample divisor $H$ on $X$,
there is a constant $c > 0$
(depending on $H$) such that, for every $0 < \varepsilon << 1$,
one can find a divisor $\Delta'$ on $X$
with $\Delta' \sim_{\Q} \Delta + \varepsilon cH$ and $(X, \Delta')$ klt.
We choose $\varepsilon$ small enough such that $\ell$ is still $(K_X + \Delta')$-negative.
Now \cite{Ka} applies and $E$ is covered by $g$-contractible rational curves $\ell_t$
with $- \ell_t . (K_X + \Delta') \le 2d$.

Suppose further that $d = 1$. Then these curves $\ell_t$
are irreducible components of the fibres of $g_{|E} : E \to g(E)$.
Hence we may assume that this family $\{\ell_t\}$ is independent of the $\varepsilon$.
Let $\varepsilon \to 0$. The lemma follows.
\end{proof}

We remark that Lemma~\ref{length}  holds without
the assumption that $d=1$. This can be extracted from
\cite[Theorem 10-3-1]{Matsuki} after
replacing the definition of dlt there by our more standard one,
which is that in \cite{KM}, and
using \cite[Theorem 1-2-5]{KaMaMa}
to drop the $\Q$-factorial assumption there.

\begin{setup}\label{Pf1}
{\rm
{\it We now prove Theorem \ref{ThF}}.
Let $(X, D + \Gamma)$ be as in Theorem \ref{ThF}.
The case $n = \dim X = 1$ is clear.
Thus we may assume that $n \ge 2$.
We proceed by induction on $n$.
By the cone theorem for $(X, D + \Gamma)$ (see \cite{Am}, \cite[Theorem 1.1]{Fu2}),
the closed cone $\NE(X)$ is generated by the closed cone $\NE(X)_{K_X + D + \Gamma \ge 0}$
and countably many
extremal rays
$R = \R_{> 0}[\ell]$ where $\ell$ is a rational curve
with $-\ell . (K_X + D + \Gamma) \in (0, 2n]$. Further, there is a contraction
$f : X \to X_1$ of an extremal ray $\R_{> 0}[\ell]$. Let
$\Exc(f) \subseteq X$ be the $f$-exceptional locus.
Let $\lfloor \Gamma \rfloor$ be the integral part of $\Gamma$.

If $\ell \subseteq G$ for an irreducible component $G$ of $D + \lfloor \Gamma \rfloor$, then
$0 > \ell . (K_X + D + \Gamma) = \ell . (K_X + D + \Gamma)_{| G}$.
By the induction on $n$ and the adjunction in Lemma \ref{restr}, the class $[\ell]$ (on $G$) is
parallel to $[\ell'] + [\ell'']$ where $[\ell']$
is a $(K_X + D + \Gamma)_{| G}$-non negative effective class on $G$ and $\ell'' \subset G$ is a
$\mathcal{D}$-rational curve
with $-\ell'' . (K_X + D + \Gamma)_{| G}$ in $(0, 2(n-1)]$
(and even in $(0, 2)$ when $\ell''$ is not $\mathcal{D}$-{\it compact}).
Since $\R_{> 0}[\ell]$ is extremal, $[\ell]$ on $X$ is parallel to (the image on $X$ of) $[\ell'']$.
We are done.

Therefore, we may add the extra assumption that
$\ell \not\subseteq (D + \lfloor \Gamma \rfloor)$ for any $[\ell] \in R$,
and hence $E_i \not\subseteq (D + \lfloor \Gamma \rfloor)$
for every irreducible component $E_i$ of $\Exc(f)$.
Since $\ell \not\subseteq D$,
$0 > \ell . (K_X + D + \Gamma) \ge \ell . (K_X + \Gamma)$.
Hence $R$ is also a $(K_X + \Gamma)$-negative extremal ray, and $(X, \Gamma)$ is dlt;
see \cite[Corollary 2.39]{KM}
(to be used later).

Since $D$ is Cartier, $D_{| E_1}$ is an effective Cartier divisor on $E_1$.
If $D = 0$ or if $D \cap E_1 = \emptyset$, then
the $\ell_t$ in Lemma \ref{length} is a $(K_X + D + \Gamma)$-negative
$\mathcal{D}$-compact rational curve.
If $f_{| E_1 \cap D}$ contracts a curve,
also denoted by $\ell$, to a point on $X_1$, then $[\ell] \in R$ and
$\ell \subseteq D$,
contradicting to our extra assumption.
Thus, $f_{| E_1 \cap D}$ is a finite morphism. Hence
$\dim f(E_1) \ge \dim f(E_1 \cap D) = \dim E_1 \cap D = \dim E_1 - 1$,
so $\dim f(E_1) = \dim E_1 - 1$.

By Lemma \ref{length},
$E_1$ is covered by $f$-contractible (extremal) rational curves $\ell$ such that
$-\ell . (K_X + \Gamma) \le 2$.
Now $0 > \ell . (K_X + D + \Gamma) \ge  -2 + \ell . D$.
Hence, if $\nu : \widetilde{\ell} \to \ell$ is the
normalization, then $2 > \ell . D = \nu^*(D_{| \ell})$. So, $D$ being Cartier and integral,
the normalization of $\ell \setminus D$
contains $\C$.
Thus $\ell$ is a
$\mathcal{D}$-rational curve.
Further, $$0 < -\ell . (K_X + D + \Gamma) \le 2 - \ell . D \le 2$$
(the latter being $< 2$ when $\ell$ is not $\mathcal{D}$-compact).
This proves Theorem \ref{ThF}.
}
\end{setup}

\begin{setup}\label{Pf2}
{\rm
{\it We now prove Theorems \ref{ThE} and \ref{ThD}}.
We proceed by induction on $n = \dim X$. The case $n = 1$ is clear.
Thus we may assume that $n \ge 2$ and $D = \sum_i D_i$ is the decomposition into
Cartier integral divisors.
By the argument in \S\ref{Pf1}, we may assume that $K_X + D + \Gamma$ is nef
in both Theorems \ref{ThE} and \ref{ThD}.

\begin{lemma}\label{amp}
Let the assumptions be as in Theorems $\ref{ThE}$ and $\ref{ThD}$,
assume the validity of these theorems
in dimension $n-1$ and of Theorem $\ref{ThE}$ in dimension $3$.
Suppose that $D_1$ is ample $($true if $n > r$ or in the case of Theorem $\ref{ThD} (1) )$.
Then $K_X + D + \Gamma$ is ample.
\end{lemma}

\begin{proof}
Suppose on the contrary that $K_X + D + \Gamma$ is not ample.
Then, by Kleiman's ampleness criterion, there is a nonzero class
$[\ell]$ in the closure $\NE(X)$ of effective $1$-cycles on $X$
(with $\R$-coefficients) such that $\ell . (K_X + D + \Gamma) = 0$.
Since $D_1 \subseteq D$ is ample,
write $D = D_{\varepsilon} + \Delta_{\varepsilon}$
with an ample $\Q$-Cartier divisor $\Delta_{\varepsilon} = \varepsilon_1 D_1 + \varepsilon_2 D$
for some $\varepsilon_i \in (0, 1)$. Now $(X, D_{\varepsilon} + \Gamma)$ is
still dlt (see \cite[Corollary 2.39]{KM}).
Note that
$$0 = \ell . (K_X + D + \Gamma) =
\ell . (K_X + D_{\varepsilon} + \Gamma) + \ell . \Delta_{\varepsilon} >
\ell . (K_X + D_{\varepsilon} + \Gamma) .$$
By the cone theorem in \cite{Am} or \cite[Theorem 1.1]{Fu2}, $\ell$ is parallel to $\ell' + \ell''$
for some class $[\ell']$ in $\NE(X)$ and a $(K_X + D_{\varepsilon} + \Gamma)$-negative
extremal rational curve $\ell''$.
Note that the nef divisor $K_X + D + \Gamma$ is perpendicular to $\ell$ and hence
$$0 = \ell'' . (K_X + D + \Gamma) = \ell' . (K_X + D + \Gamma).$$

Let $g : X \to X_2$ be the contraction of
the $(K_X + D_{\varepsilon} + \Gamma)$-negative extremal ray $\R_{> 0}[\ell'']$.
If $\ell''$ lies in
an irreducible component $G$ of $D + \lfloor \Gamma \rfloor$, then
$0 = \ell'' . (K_X + D + \Gamma) = \ell . (K_X + D + \Gamma)_{| G}$,
which contradicts the ampleness result in lower dimension by the inductive assumption
as $\dim G < \dim X$
(see Lemma \ref{restr}
and note that, when $n=4$, we reduce to Theorem \ref{ThE}(2)
for all cases of Theorems \ref{ThE} and \ref{ThD}).
Thus we may assume that
$\ell'' \not\subseteq (D + \lfloor \Gamma \rfloor)$.
So $\ell'' \cap (D + \lfloor \Gamma \rfloor)$ is a finite set and
no irreducible component $E_1$ of
$\Exc(g)$ is contained in $D + \lfloor \Gamma \rfloor$.
Since we may assume to reach a contradiction that $X$ has no $(K_X + D + \Gamma)$-non-positive $\mathcal{D}$-rational curves,
$\ell'' \cap D$ is a non-empty finite set.

Thus $0 = \ell'' . (K_X + D + \Gamma) > \ell'' . (K_X + \Gamma)$.
Hence $\R_{> 0}[\ell'']$ is also a $(K_X + \Gamma)$-negative extremal ray.
By the argument in \S \ref{Pf1} and noting that $\ell'' \not\subseteq D$,
$$\dim g(E_1) \ge \dim g(E_1 \cap D) = \dim E_1 \cap D = \dim E_1 - 1 .$$
By Lemma \ref{length},
$E_1$ is covered by rational curves $\ell''$ with
$-\ell'' . (K_X + \Gamma) \le 2$.
Moreover $0 = \ell'' . (K_X + D + \Gamma) \ge  -2 + \ell'' . D$.
Thus $\ell'' . D \le 2$, and
$\ell''$ is a $(K_X + D + \Gamma)$-non-positive
$\mathcal{D}$-rational curve or $\mathcal{D}$-algebraic $1$-torus as in \S \ref{Pf1}, a contradiction!
\end{proof}

By Lemma \ref{amp} and the assumption of Theorems \ref{ThE} and \ref{ThD},
we may assume that $n = r$, $D \ne 0$ and
$D$ is Cartier. This also includes the case of Theorem \ref{ThE} (2) where $n = 3$.
(The case $n = 2$ with $D = 0$ is taken care of
via Remark \ref{ABH}
by the same argument given below).
Since we have proved the nefness of $K_X + D + \Gamma$,
by the abundance assumption for dlt pairs of dimension $n = r$,
there is a fibration
$h : X \to Z$ with connected general fibre $F$ such that
$$K_X + D + \Gamma = h^*H \,\,\,\mbox{and}\,\,\,
(K_X + D + \Gamma)_{| F} = h^*H _{| F} \sim_{\Q} 0 $$
for some ample $\Q$-divisor $H$ on $Z$.
We divide into three cases:
Case (I) $\dim Z = 0$, Case(II) $0 < \dim Z < \dim X$ and
Case (III) $\dim Z = \dim X$.

\par \vskip 1pc
Case(I) $\dim Z = 0$. Then $K_X + D + \Gamma \sim_{\Q} 0$.
Thus for an irreducible component $G$ of $D$, we have
$(K_X + D + \Gamma)_{| G} \sim_{\Q} 0$.
{\it
There is such a $G$ if
$D$ has at least $n-r+1 \ge 1$ Cartier components as in
Theorem $\ref{ThE}$ or $\ref{ThD}$.}

Applying the adjunction as in Lemma \ref{restr} and replacing $X$ by
$G$ (of dimension $n - 1 = r - 1$)
and further by an irreducible component of some $D_1 \cap D_2 \cap \dots$,
we may assume that either $\dim X \le 1$
which is clear; or $D = 0$, $(X, \Gamma)$ is dlt and $K_X + \Gamma \sim_{\Q} 0$.

Consider the latter case.
If $\Gamma \ne 0$, then $X$ is uniruled and every rational curve $C$ on $X$
is a $\mathcal{D}$-compact rational curve with
$(K_X + D + \Gamma)_{| C} \sim_{\Q} 0$, which is the excluded case.

If $\Gamma = 0$, then $K_X \sim_{\Q} 0$ and hence
$X$ is $\mathcal{D}$-compact variety which is $\Q$-CY.

{\bf N.B.}
If the initial pair $(X, D + \lfloor \Gamma \rfloor)$ is log smooth,
then $X$ (and subsequent namesakes of the pairs) are (log) smooth. We now have
$K_X \sim_{\Q} 0$. By the Beauville-Bogomolov decomposition,
either $X$ is a $\Q$-torus, or it has a Calabi-Yau manifold as its finite \'etale cover.
Thus CY($m$) (with $m< r = n$) implies, in the latter case,
the existence of a
$\mathcal{D}$-compact rational curve $C$ on the initial $X$ with
$(K_X + D + \Gamma)_{| C} \sim_{\Q} 0$.

\par \vskip 1pc
Case(II) $0 < \dim Z < \dim X$.
Then
$K_F + (D + \Gamma)_{| F} = (K_X + D + \Gamma)_{| F} = h^*H _{| F} \sim_{\Q} 0$
as $F$ is a fibre of $h$.
So Theorems \ref{ThE} and \ref{ThD}
follow by the same arguments
as for Case~(I).

\par \vskip 1pc
Case(III) $\dim Z = \dim X$. Then $K_X + D + \Gamma$ is nef and big.
We may assume that $h : X \to Z$ is birational but not isomorphic,
$\emptyset \neq \Exc(h) \subset X$ the exceptional locus of $h$, and
$\ell \subseteq \Exc(h)$ a curve contracted by $h$.
Then $(K_X + D + \Gamma)_{| \ell} = h^*H_{| \ell} \sim_{\Q} 0$.

If $\ell \subseteq G$ for an irreducible component $G$ of $D + \lfloor \Gamma \rfloor$,
then $(K_X + D + \Gamma)_{|G}$ is not ample and we
use the abundance assumption on $G$ and divide into Cases (I) - (III) again
to conclude Theorems \ref{ThE} and \ref{ThD} by the induction on dimension.

Thus we may assume that $\ell \not\subseteq (D + \lfloor \Gamma \rfloor)$.
Hence $\ell \cap (D + \lfloor \Gamma \rfloor)$ is a finite set and $\ell . D \ge 0$.
So, no irreducible component of
$\Exc(h)$ is contained in $D + \lfloor \Gamma \rfloor$.

Consider first the case that $\ell . (K_X + \Gamma) < 0$
for some $h$-contractible curve
$\ell$. Then $\ell$ is parallel to $\ell' + \ell''$
for some class $[\ell'] \in \NE(X)$ and
a $(K_X + \Gamma)$-negative extremal rational curve $\ell''$
by the cone theorem (see \cite[Theorem 1.1]{Fu2}).
Note that the nef divisor $K_X + D + \Gamma$ is perpendicular to $\ell$ and hence\\[-5mm]
$$0 = \ell'' . (K_X + D + \Gamma) = \ell' . (K_X + D + \Gamma) .$$
Thus $\ell'' \not\subseteq (D + \lfloor \Gamma \rfloor)$ by the same argument
as above for $\ell$ (for later use).

Let  $h_1 : X \to Z_1$
be the contraction of the extremal ray $\R_{> 0}[\ell'']$.
Since $\ell'' . h^*H = 0$ with $H$ ample, every irreducible component
$E_1$ of the exceptional locus of $h_1$ is a subset of $\Exc(h)$
and hence is not contained in $D + \lfloor \Gamma \rfloor$.
As in the proofs of Lemma \ref{amp} and \S \ref{Pf1} and by Lemma \ref{length},
$E_1$ is covered by rational curves $\ell''$ with
$-\ell'' . (K_X + \Gamma) \le 2$.
Now $0 =$ $\ell'' . (K_X + D + \Gamma) \ge  -2 + \ell'' . D$.
So $\ell'' . D \le 2$ and Theorems \ref{ThE} and \ref{ThD} hold
as in \S \ref{Pf1}.

Consider now the case that $\ell . (K_X + \Gamma) \ge 0$ for every curve $\ell$ on $X$ contracted by $h$.
Then $0 = \ell . (K_X + D + \Gamma) \ge \ell . D \ge 0$. So $\ell . (K_X + \Gamma) = \ell . D = 0$ and
$\ell \cap D = \emptyset$. Thus $\Exc(h) \cap D = \emptyset$.
By the construction of the birational morphism $h : X \to Z$,
$K_X + D + \Gamma = h^*(K_Z + h_*(D + \Gamma))$. So when $h$ is a small contraction,
a `good log resolution' for the dlt pair $(X, D + \Gamma)$ as in \cite[Theorem 2.44(2)]{KM}
is also a `good log resolution'
for $(Z, h_*(D + \Gamma))$ and hence the latter is also dlt.
Thus
by \cite[Corollary 1.5]{HM}, every fibre of $h : X \to Z$ is rationally chain connected.
Hence the above $\ell$ can be chosen to be a rational curve away from $D$,
and is a $\mathcal{D}$-compact rational curve with
$(K_X + D + \Gamma)_{| \ell} = h^*H_{| \ell} \sim_{\Q} 0$, which is the excluded case.

Therefore, we may assume that $h$ contracts an irreducible divisor $S \subseteq \Exc(h)$.
By Lemma \ref{divcontr} below, a general fibre $F_S$ of $h_{|S} : S \to h(S)$ is uniruled.
Since $S \cap D \subseteq \Exc(h) \cap D = \emptyset$, every
rational curve ($F \supseteq$) $\ell \subseteq F_S \subseteq S$
is a $\mathcal{D}$-compact rational curve
with $(K_X + D + \Gamma)_{|\ell} \sim_{\Q} 0$,
which is the excluded case.
This proves Theorems \ref{ThE} and \ref{ThD},
modulo Lemma \ref{divcontr} below applied to $K_V + B = K_X + D + \Gamma$.
}
\end{setup}

\begin{lemma}\label{divcontr}
Let $V$ be a normal projective variety and $B$ an effective Weil $\Q$-divisor
with coefficients $\le 1$ such that $K_V + B$ is $\Q$-Cartier.
Let $\varphi : V \to W$ be a birational morphism and $S\subset V$ an irreducible divisor.
Suppose that $\varphi$ contracts $S$,
$S \not\subseteq \lfloor B \rfloor$
and $C . (K_V + B) \le 0$ for a general curve $C$ on $S$ contracted by $\varphi$.
Then a general fibre $F_S$ of $\varphi_{|S} : S \to \varphi(S)$ is a uniruled variety.
\end{lemma}

\begin{proof}
Let $\sigma_1 : V' \to V$ be a dlt blowup of the pair $(V, B)$
(see \cite[Theorem 10.4]{Fu2}) with
$E_{\sigma_1}$  the reduced exceptional divisor and
$B'$ the sum of $E_{\sigma_1}$ and the proper transform
$\sigma'B$ of $B$. Then the pair $(V', B')$ is $\Q$-factorial
and dlt and we have
$$K_{V'} + B' + E' = \sigma_1^*(K_V + B).$$
for a $\sigma_1$-exceptional effective divisor $E'$.
Let $S' := \sigma_1' S$ be the proper transform of $S$, which is now $\Q$-Cartier.
Write $B' := a S'  + \widetilde{B'}$ for some $a \in [0, 1)$ such that $S'$
is not a component of $\widetilde{B'}$.
Let $\sigma_2 : V'' \to V'$ be a dlt blowup of the pair $(V', \widetilde{B'} + S')$.
Thus, if $E_{\sigma_2}$ is the
reduced exceptional divisor and $\widetilde{B''}$ the sum
of $E_{\sigma_2}$ and $\sigma_2' \widetilde{B'}$,
then there is a $\sigma_2$-exceptional divisor $E'' \ge 0$ such that
the pair $(V'', \widetilde{B''} + S'')$ is $\Q$-factorial dlt
with $S'' := \sigma_2' S'$
and such that
$$K_{V''} + \widetilde{B''} + S'' + E'' =
\sigma_2^*(K_{V'} + \widetilde{B'} + S').$$
Let
$C' \subset S'$ and $C'' \subset S''$
be the strict transforms of the general curve $C \subset S$ (contracted by $\varphi$).
By the adjunction formula given in Lemma \ref{dlt},
we have a $\Q$-Cartier adjoint divisor $(K_{V''} + \widetilde{B''} + S'')_{| S''} = K_{S''} + \Delta$
on the normal variety $S''$ for some $\Delta \ge 0$
and that the pair $(S'', \Delta)$ is dlt.
Thus
$$\begin{aligned}
&C'' . (K_{S''} + \Delta) \le C'' . (K_{V''} + \widetilde{B''} + S'' + E'')_{| S''} = \\
& C'' . \sigma_2^*(K_{V'} + \widetilde{B'} + S')_{| S''}  =
C' . (\sigma_1^*(K_V + B) - E') + (1-a) C' . S' \\
&\le C . (K_V + B) + (1-a) C' . S' \le (1-a) C' . S' .
\end{aligned}$$
Since the composition $V' \to V \to W$ is birational and contracts $S'$,
by the well-known negativity lemma (see e.g. \cite[Lemma 3.6.2]{BCHM}),
one can choose general curves $C'$ such that $C' . S' < 0$
and that a general fibre $F_{S'}$ of the fibration
$\varphi_{S'}$ on $S'$
(resp.~$F_{S''}$ of $\varphi_{S''}$ on $S''$)
induced by $\varphi_{|S}$ is covered by such curves $C'$ (resp.~$C''$).
Thus $C'' . (K_{S''} + \Delta) < 0$.

Since $(S'', \Delta)$ is dlt, for a general fibre $F_{S''}$ of the fibration $\varphi_{S''}$,
we have $(K_{S''} + \Delta)_{|F_{S''}} = K_{F_{S''}} + \Theta$ with
$(F_{S''}, \Theta)$ dlt. Now $C'' . (K_{F_{S''}} + \Theta) < 0$ with $C'' \subseteq F_{S''}$
covering $F_{S''}$.
Thus $F_{S''}$ (and hence $F_S$) are uniruled.
Indeed, let $\tau : \widetilde{F} \to F_{S''}$ be a resolution
and write
$$K_{\widetilde F} + \tau' \Theta + E' = \tau^*(K_{F_{S''}} + \Theta) + E''$$
with $\tau$-exceptional effective divisors $E'$, $E''$.
Since $C'' . (K_{F_{S''}} + \Theta) < 0$ for curves $C''$ covering $F_{S''}$,
$K_{F_{S''}} + \Theta$ on $F_{S''}$ and
$\tau^*(K_{F_{S''}} + \Theta)$ on $\widetilde{F}$
are not pseudo-effective (see \cite[Theorem 0.2]{BDPP}).
Thus $K_{\widetilde F}$ is not pseudo-effective,
for otherwise the left hand side of
the displayed equation above
would be pseudo-effective, and pushing forward the right side
would give the pseudo-effectivity of $K_{F_{S''}} + \Theta$, a contradiction.
Therefore, $\widetilde{F}$ (and hence its image $F_S$ on $S$)
are uniruled by the well-known uniruledness criterion of Miyaoka-Mori.

This completes the proof of Lemma \ref{divcontr} and also
of Theorems \ref{ThE} and \ref{ThD}.
\end{proof}

\begin{setup}\label{PE}
{\rm
{\it Now we prove Theorem \ref{ThH}.}
Set $n := \dim X$.
Since $(X, D)$ is dlt, \cite[Proposition 2.43]{KM} implies
that $(X, D_t)$ is klt for some $D_t \sim_{\Q} D + t H$
with $H$ an ample divisor and $t \to 0^+$.
Since $K_X + D$ is not pseudo-effective,
the $H$-directed LMMP in \cite[Proof of Corollary 1.3.3]{BCHM}
(the $X$ there and hence here are assumed to be $\Q$-factorial)
implies the existence of a composition
$$X = X_0 \overset{f_0}\dasharrow \cdots \overset{f_{m-1}}\dasharrow X_m =: W$$
of divisorial contractions and flips of $(K_{X_i} + D(i))$-negative extremal rays
with $D(i)$ the push-forward of $D$
and the existence of a Fano contraction
$\gamma: W \to Y$ of a $(K_W + B)$-negative extremal ray
$\R_{> 0}[\ell]$ with $F$ a general fibre, $B := D(m)$.
Note that $(X_i, D(i))$ is still $\Q$-factorial and dlt.
Also $-(K_W + B)$ is relatively ample over $Y$ and
$-(K_W + B)_{|F} = -(K_F + B_{|F})$ is ample
by the definition of Fano contraction.
The divisors contracted by the birational map
$f = f_{m-1} \circ \cdots \circ f_0 : X \dasharrow W$
being uniruled (\cite[Corollary 5-1-4]{KaMaMa}),
no component of $D$ is contracted by $f$ by assumption.
Let $B = \sum B_i$ be the irreducible decomposition.
It follows that no $B_i$ is uniruled and $B$ has the same number
of irreducible components as $D$.
By Lemma \ref{dlt}, each $B_i$ is normal and
for some $\Delta_i \ge 0$, 
$$(K_W + B)_{|B_i} = K_{B_i} + (B - B_i)_{|B_i} + \Delta_i .$$

Suppose that $\gamma_{|B_1} : B_1 \to \gamma(B_1)$ is not generically finite
so that its general fibre can be taken as the extremal curve $\ell$.
Since $-(K_W + B)$ is relative ample over $Y$,
we have
$$0 > \ell . (K_W + B) = \ell . (K_W + B)_{|B_1} \ge \ell . (K_{B_1} + \Delta_1) .$$
So $K_{B_1} + \Delta_1$ is not pseudo-effective.
Hence $B_1$ is uniruled, contradicting the assumption.
Thus, we may assume that $\gamma_{|B_i} : B_i \to \gamma(B_i)$ is generically finite
for every $i$. In particular,
$n-1 \ge \dim Y \ge \dim \gamma(B_i) = \dim B_i = n-1$.
So $\dim Y = n-1 = \dim \gamma(B_i)$ and $\gamma(B_i) = Y$.
Also $\dim F = 1$ and hence $F \cong \PP^1$.
The ampleness of $-(K_F + B_{|F})$ implies
that $|B \cap F| = 1$. So $B = B_1$, i.e.,
$B$ (and therefore $D$) is irreducible,
since all the $B_i$'s dominate $Y$.

Let $\Sigma \subset X_m = W$ be the image of the union of
the exceptional
locus and the indeterminacy locus of all $f_i$ ($0 \le i < m$),
i.e., the subset of $W$ over which $f_{m-1} \circ \cdots \circ f_0$
is not an isomorphism.
Then
$\dim \Sigma \le n-2$. So $\gamma(\Sigma)$
is a proper subset of $Y$.
Thus
$U_m := X_m \setminus (D(m) \cup \Sigma)$
is covered by $F \setminus F \cap D(m) \cong \C$ for a general fibre $F$ of $\gamma$.
Regarding $U_m$ as an open subset of $X \setminus D$,
$X$ is dominated by $\SD$-rational curves ($F \cong$) $F_0 \subset X$
which are the strict transforms of general $F \subset X_m$ with $F \cap \Sigma = \emptyset$.
Further, $-F_0 . (K_X +D) = - F. (K_{X_m} + D(m)) = 2- 1 = 1$.
Theorem \ref{ThH} is proved.
}
\end{setup}

\par \vskip 1pc
\begin{remark}\label{ABH}
By the surface theory, a $\Q$-CY surface
has either a rational curve or smooth elliptic curves, or it
is a simple abelian surface or a
a normal $K3$ surface with $s \ge 1$ singularities $($all Du Val$)$
and infinitely many (singular) elliptic curves (see \cite{MM}).
In particular, $S$ is not Brody hyperbolic.
\end{remark}

When the boundary $D$ is not Cartier, the
steps in \S \ref{Pf1} $\sim$ \S \ref{Pf2} and Remark \ref{ABH} imply:

\begin{remark}\label{PropD}
Let $X$ be a projective variety of dimension
$n \le 2$, $D$ a reduced Weil
divisor and
$\Gamma$ an effective Weil $\Q$-divisor such that the pair $(X, D + \Gamma)$ is dlt.
\begin{itemize}
\item[(1)]
Suppose that $K_X + D + \Gamma$ is not nef.
Then there is a $(K_X + D + \Gamma)$-negative $\mathcal{D}$-rational curve.

\par \noindent
In particular, if the pair $(X, D)$ is Mori hyperbolic, then $K_X + D + \Gamma$ is nef.
\item[(2)]
Suppose that $K_X + D + \Gamma$ is not ample
and has strictly positive degree on every $\mathcal{D}$-rational curve
and $\mathcal{D}$-algebraic $1$-torus. Then either
$(K_X + D + \Gamma)_{| \ell} \sim_{\Q} 0$ for
a $\mathcal{D}$-compact curve $\ell$ which is smooth elliptic, or
$D = \Gamma = 0$ and $X$ is either a simple abelian surface
or a normal $K3$ surface with $s \ge 1$ singularities $($all Du Val$)$
and infinitely many (singular) elliptic curves.

\par \noindent
In particular, if the pair $(X, D)$ is Brody hyperbolic,
then $K_X + D + \Gamma$ is ample.
\end{itemize}
\end{remark}

\section{Appendix: Proof of Remark \ref{PropD}}

In this appendix, we prove Remark \ref{PropD}, and give a corollary.
Set $n := \dim X$.
When $n = 1$, it is clear. Assume
that $n = 2$.
Since $(X, D + \Gamma)$ is dlt, so is $(X, D)$; $X$ is $\Q$-factorial and klt
(see \cite[Propositions 4.11 and 2.41, Corollary 2.39]{KM}).

Assume that $K_X + D + \Gamma$ is not nef. By the cone theorem \cite[Theorem 3.4.7]{KM},
there is a $(K_X + D + \Gamma)$-negative extremal rational curve $\ell$.
Let $\xi : X \to Z$ be the
corresponding contraction with $F$ a general fibre.
Then $\xi$ is either a Fano or a divisorial contraction.

Consider first the case that $\xi$ is a Fano contraction.
Then $-(K_X + D + \Gamma)_{| F}$ is ample.
Thus by the adjunction formula of Lemma \ref{dlt},
if $\dim Z = 1$ (resp.~$\dim Z = 0$
and $D \ne 0$),
we have anti-ample divisor
$(K_X + D + \Gamma)_{| F} = K_F + (D + \Gamma)_{| F}$
$\big($~resp.~of $(K_X + D + \Gamma)_{| G} = K_{G} + (D - G)_{| G}  + \Delta$
with $\Delta \ge 0$,
$G$ an irreducible component of $D$ $\big)$, so $F$ (resp.~$G$) is
a $\mathcal{D}$-rational curve having negative intersection with $K_X + D + \Gamma$.
When $\dim Z = 0$ and $D = 0$, the ampleness of $-(K_X + D + \Gamma)$ implies that
$X$ is uniruled, and hence every rational curve $C$ on $X$ is a
$\mathcal{D}$-compact rational curve with $C . (K_X + D + \Gamma) < 0$.

Consider the remaining
case that $\xi$ is a divisorial (i.e., birational) contraction.
Since $(X, D + \Gamma)$ is dlt, so are $(Z, \xi_*(D + \Gamma))$ and $(Z, \xi_*D)$.
We have the following two cases.

If the rational curve $\ell$
is contained in a component of $D$, written $\ell \subseteq D$
by abuse of notation, then, by Lemma \ref{dlt},
it would be normal and
(transversally) meet the other components of $D$ at at most one point
by the calculation:\\[-5mm]
$$\begin{aligned}
(*) \,\,\,\, 0 > &\ell . (K_X + D + \Gamma) = \deg (K_X + D + \Gamma)_{| \ell} =\\
&\deg(K_{\ell} + (D - \ell)_{| \ell} + \Delta) \ge \deg(K_{\ell} + (D - \ell)_{| \ell})
= -2 + \ell . (D - \ell).
\end{aligned}$$
Hence $\ell$ is a
$\mathcal{D}$-rational curve having negative intersection with $K_X + D + \Gamma$.
Indeed, recall that $X$ is $\Q$-factorial, so the intersections here
make sense. Also, every intersection point $P$ of $\ell$ with other components of $D$ is
in the smooth locus of $X$ and $D$ is of normal crossing at $P$
by the classification of dlt singularities $(X, D)$ (see \cite[Theorem 4.15 (1)]{KM}).

Consider now the case that $\ell \not\subseteq D$.
A simple check via the short but full classification of
the dlt pair $(Z, \xi_*D)$ by Alexeev \cite[\S 3]{Ko}
or equivalently \cite[Theorems 4.7 and 4.15]{KM}
shows that
$\ell$ is a rational curve homeomorphic to $\PP^1$
(whose inverse on the minimal resolution of $X$ is a tree of smooth rational curves)
which meets $D$ at at most two points
with the case of two points
occurring only when $P_Z := \xi(\ell)$ is a smooth point of $Z$
at which $\xi_*(D + \Gamma)$ is of normal crossing
(see Case (1) of \cite[Theorem 4.15]{KM}).
By Claim \ref{2pt} below, the $(K_X + D + \Gamma)$-negative
curve $\ell$ can not meet $D$ at two points.
Thus $\ell$ meets $D$ at at most one point, so
$\ell$ is a $\mathcal{D}$-rational curve.
Remark \ref{PropD} (1) is proved.

\begin{claim}\label{2pt}
Suppose that $\ell$ meets $D$ at two distinct points $P_i$.
Then $\ell . (K_X + D + \Gamma) \ge 0$.
\end{claim}

We prove Claim \ref{2pt}.
Consider the case that $\ell$ meets $D$ at two distinct points $P_i$.
Then $P_Z = \xi(\ell)$ is a smooth point of $Z$
and also a simple node of $D_Z := \xi(D)$.
Let $\delta : X'' \to X$ be the minimal desingularization of the set $\{P_1, P_2\} \cap (\Sing X)$,
and $E_i$ the $\delta$-exceptional divisor lying over $P_i$.
For each $i$, we set $E_i = 0$ if $\delta$ is isomorphic over $P_i$, i.e., $P_i$ is a smooth point of $X$,
otherwise, $E_i$ is a connected divisor.
We set $\delta = \id$ if both $P_i$ are smooth points of $X$.
Let $D'' := \delta' D$ and $\ell'' := \delta' \ell$ be the proper transforms.
Since $\xi(\ell)= P_Z$ is a smooth point on $Z$ and hence
$\xi \circ \delta : X'' \to Z$ is the blowup of this smooth point,
the $\xi \circ \delta$-inverse of this smooth point,
i.e., the $\xi \circ \delta$-exceptional divisor $E_{\xi \circ \delta}$
is equal to $\ell'' + E_1 + E_2$,
a simple normal crossing divisor and also a union of smooth rational curves,
whose dual graph is a connected tree. Since $\ell''$ is the only $(-1)$-curve
in $E_{\xi \circ \delta} = \ell'' + E_1 + E_2$ connecting $E_1$ and $E_2$
($\delta : X'' \to X$ being minimal)
and $\ell''$ meets $D''$ at the (unique) point lying over $P_i$ when $E_i = 0$
(for one or both of $i \in \{1, 2\}$),
i.e., when $\delta$ is isomorphic over $P_i$,
the blowup $\xi \circ \delta : X'' \to Z$ produces a linear chain
as $E_{\xi \circ \delta} = \ell'' + E_1 + E_2$ which meets $D''$ exactly at two points
(and transversally). In other words,
$\xi \circ \delta$ is the composition of the blowup of the node $P_Z$ of $D_Z$
followed by the blowups of nodes on the total transforms of $D_Z$.
Inductively, we can verify that
$K_{X''} + D'' + \ell'' + E_1 + E_2 = (\xi \circ \delta)^*(K_Z + D_Z)$.
Pushing this forward via $\delta_*$, we get
$K_X +  D + \ell = \xi^*(K_Z + D_Z)$, which has zero intersection with
the $\xi$-exceptional curve $\ell$ (with $\ell^2 < 0$).
Write $\Gamma = b\ell + \Gamma_1$ where $b \in [0, 1]$
and $\ell$ and $\Gamma_1$ have no common component.
Then $0 > \ell . (K_X + D + \Gamma) = \ell . (K_X + D + \ell + (b-1) \ell + \Gamma_1)
= (1-b)(- \ell^2) + \ell . \Gamma_1 \ge 0 + 0$.
This is a contradiction.
This proves Claim \ref{2pt}.

\par \vskip 1pc
We return to the proof of Remark \ref{PropD}.
We may now assume that $K_X + D + \Gamma$ is nef.
By the abundance theorem for dlt pairs and $n \le 3$ (see \cite{KeMaMc}),
there is a morphism $\sigma : X \to Y$ with a connected general fibre $F$ and
an ample divisor $H$ on $Y$ such that $K_X + D + \Gamma = \sigma^*H$.
We have $(K_X + D + \Gamma)_{| F} = \sigma^*H_{| F}$ $\sim_{\Q} 0$.
We divide into: Case (I) $\dim Y = 0$, Case (II) $\dim Y = 1$ and Case (III) $\dim Y = 2$.

\par \vskip 1pc
Case (I) $\dim Y = 0$, i.e., $K_X + D + \Gamma \sim_{\Q} 0$.
If $G$ is an irreducible component of $D$, then
$(K_X + D + \Gamma)_{| G} \sim_{\Q} 0$ and $G$
is a $\mathcal{D}$-rational curve, a
$\mathcal{D}$-algebraic $1$-torus or a $\mathcal{D}$-torus.

Thus we may assume that $D = 0$ and $K_X + \Gamma \sim_{\Q} 0$.
If $\Gamma \ne 0$, then $X$ is uniruled and every rational curve on $X$
is a $\mathcal{D}$-compact rational curve
having zero intersection with $K_X + D + \Gamma$.
Suppose that $\Gamma = 0$ so that $K_X \sim_{\Q} 0$.
Let $X'' \to X$ be the global index-one cover
(unramified in codimension-one)
so that $K_{X''} \sim 0$
and $X''$ has at worst canonical singularities.
The minimal resolution $\hat{X}$ of $X''$ satisfies
$K_{\hat X} \sim 0$.
Thus Remark \ref{PropD} follows from the surface theory
(see Remark \ref{ABH} below).
Indeed, a (smooth) K3 surface has at least one rational curve and
infinitely many elliptic curves (see \cite{MM}).

\par \vskip 1pc
Case (II) $\dim Y = 1$. Then $(K_X + D + \Gamma)_{| F} \sim_{\Q} 0$.
Hence a general fibre $F$ is
either a $\mathcal{D}$-rational curve, or a
$\mathcal{D}$-algebraic $1$-torus, or a $\mathcal{D}$-torus.

\par \vskip 1pc
Case (III) $\dim Y = 2$, i.e.,
$\sigma$ is biratonal. We may assume that $\sigma$ is not isomorphic
and hence contracts a curve $\ell$.
As $K_X + D + \Gamma = \sigma^*(K_Y + \sigma_*(D + \Gamma))$,
the pair $(Y, \sigma_*(D + \Gamma))$
is also lc. Of course,
$(K_X + D + \Gamma)_{| \ell} \sim_{\Q} 0$.

Remark \ref{PropD} in this case follows from Claim \ref{lc} below,
without using powerful machineries.

Alternatively, as in \S \ref{Pf2} via LMMP, we can reduce to the case of Lemma \ref{divcontr}
and conclude that $\ell$ ($= S$ there) is a
$\mathcal{D}$-compact rational curve.
\begin{claim}\label{lc}
At least one of the following cases holds.
\begin{itemize}
\item[(C1)]
$\ell$ is a component of $D$; $\ell$ is either a $($smooth$)$ elliptic curve
disjoint from $D-\ell$, or a smooth rational curve and $($transversally$)$ meet the other components of $D$
at at most two points.

In the rest, $\ell$ is not a component of $D$.

\item[(C2)]
The geometric genus $g(\ell) \le 1$
$($with equality holding true only when $\ell$ is a $($smooth$)$ elliptic curve and
$P = \sigma(\ell)$ is a non-klt point
of $Y )$. $\ell$ is disjoint from $D$.

\item[(C3)]
$\ell$ is a rational curve homeomorphic to $\PP^1$
and meets $D$ at at most two points.

\item[(C4)]
There is a rational curve $D_1$ in $D$ which is homeomorphic
to $\PP^1$ and meets
the other components of $D$ at at most one point
$($this case needs to be added only when $P = \sigma(\ell)$ is a non-klt point
of $Y )$.

\end{itemize}
\end{claim}

We prove Claim \ref{lc}.
When $\ell$ is a component of $D$, we have case (C1)
by the calculation in (*) above,
with `$0 >$' replaced by `$0 =$'.

So assume that $\ell$ is not a component of $D$.
Let $\eta: X' \to X$ be the minimal resolution and $\Exc(\eta)$ the exceptional divisor.
Let $D' = \eta'D$, $\Gamma' = \eta' \Gamma$ and $\ell' = \eta' \ell$ be
the proper transforms of $D$, $\Gamma$ and $\ell$, respectively.
Write
$K_{X'} + E_K = \eta^*K_X$,
for some $\eta$-contractible divisor $E_K$.
Since $\eta$ is a minimal resolution, $E_K \ge 0$.
Write $\eta^*D = D'+ E_D$ and $\eta^*\Gamma = \Gamma'+ E_{\Gamma}$,
with $E_D$ and $E_{\Gamma}$ effective and $\eta$-contractible.
Then
$$K_{X'} + D' + \Gamma' + E = \eta^*(K_X + D + \Gamma)$$
where $E = E_K + E_D + E_{\Gamma}$. In particular, $\Supp(D'+ E) \supseteq \eta^{-1}(D)$.

Write $\Gamma = a \ell + \Gamma_1$ and hence $\Gamma' = a \ell' + \Gamma_1'$,
where $a \in [0, 1]$ and $\Gamma_1$ (resp.~$\Gamma_1'$) does not contain $\ell$ (resp.~$\ell'$).
Since $\ell'$ is contracted by $\eta \circ \sigma$, $(\ell')^2 < 0$.
Hence $(a-1) \ell'^2 \ge 0$.
Note that $\ell$ is perpendicular to $\sigma^*(K_Y + \sigma_*(D + \Gamma)) = K_X + D + \Gamma$.
So
$$
\begin{aligned}
0 &= \ell . (K_X + D + \Gamma) = \ell' . \eta^*(K_X + D + \Gamma) = \ell' . (K_{X'} + D' + \Gamma' + E) \\
&= \ell' . (K_{X'} + \ell' + D' + \Gamma_1' + E) + (a-1) \ell'^2 \\
&\ge 2p_a(\ell') - 2 + \ell' . (D' + \Gamma_1' + E) \ge 2p_a(\ell') - 2.
\end{aligned}$$
Thus $g(\ell) \le p_a(\ell') \le 1$. Further, if $p_a(\ell') = 1$, then the inequalities in the above
display all become equalities. So $a = 1$,
and $\emptyset = \ell' \cap (D' + E) \supseteq \ell' \cap \eta^{-1}(D)$.
This is case (C2).

We may now assume that $p_a(\ell') = 0$, i.e., $\ell' \cong \PP^1$.
Since $P = \sigma(\ell)$ is a log canonical singularity of $(Y, \sigma_*(D + \Gamma))$
and hence the reduced divisor $\sigma_*D$ contains at most two analytic branches at $P$
by \cite[Theorem 4.15]{KM}, $\ell$ meets $D$ at at most two points.
The exceptional divisor $E_P$ of the {\it minimal}
resolution of $P$ is given in \cite[Theorem 4.7, 1--5]{KM}.
The rational curve $\ell$ is homeomorphic to $\PP^1$ and hence
we are in case (C3) unless
Case~(1) in \cite[Theorem 4.7]{KM} occurs, i.e.,
$E_P$ is a rational curve with a simple node
and the connected component $\Sigma$ of $D + \ell$
containing $\ell$ is contracted to the point $P$.

We consider the latter case.
$\Sigma$ is obtained from the nodal curve $E_P$ by blowing up some points on it and then blowing
down some curves in the inverse of $E_P$. Hence $\Sigma$
is the image of a divisor $\Sigma''$ on a blowup $X''$ of $X$
such that $\Sigma''$ is of simple normal crossing and consists of smooth rational curves,
$\Sigma''$ contains only one loop (which is simple), and $X'' \to X$
is the contraction of some rational trees contained in $\Sigma''$, $X$ being klt.
Thus, either $\ell$ has a node and case (C2) or C(4) occurs, or
$\ell$ is homeomorphic to $\PP^1$ and case (C3) occurs.
This proves Claim \ref{lc} and also Remark \ref{PropD}.

\par \vskip 1pc
The result below follows from Remark \ref{PropD}
and Lemma \ref{restr}.

\begin{corollary}\label{cPropD}
Let $(X, D)$ be a projective Brody hyperbolic pair with
$\dim X \le 3$ and $\Gamma \ge 0$ a Weil $\Q$-divisor
such that the pair $(X, D + \Gamma)$ is dlt.
Then the restriction $(K_X + D + \Gamma)_{| G}$ is an ample divisor on $G$
for every irreducible component $G$ of $D + \lfloor \Gamma \rfloor$.
\end{corollary}

\vspace{.5cm}


%
%
%
%


\begin{thebibliography}{99}

\bibitem{Am}
F.~Ambro, Quasi-log varieties, Tr. \ Mat. \ Inst. \ Steklova \textbf{240} (2003), Biratsion. \
Geom. \ Linein. \ Sist. \ Konechno Porozhdennye Algebry, 220–239;
translation in Proc. Steklov Inst. Math. 2003, no. 1 (240), 214–233.

\bibitem{BCHM}
C.~Birkar, P.~Cascini, C.~D.~Hacon and J.~McKernan,
Existence of minimal models for varieties of log general type,
J. \ Amer. \ Math. Soc. \ \textbf{23} (2010) 405-468.

\bibitem{BDPP}
S.~Boucksom, J. -P.~Demailly, M.~Paun and T.~Peternell,
The pseudo-effective cone of a compact K\"ahler manifold and
varieties of negative Kodaira dimension,
J. \ Algebraic Geom. \ \textbf{22} (2013), 201-248.

\bibitem{Co}
A.~Corti et al., Flips for 3-folds and 4-folds, Oxford Lecture Ser.
\ Math. \ Appl. \ \textbf{35}, Oxford Univ. Press, Oxford, 2007.

\bibitem{Fu2}
O.~Fujino,
Fundamental theorems for the log minimal model program, Publ. \ Res. \ Inst. \ Math. \ Sci. \
\textbf{47} (2011), no. 3, 727--789.

\bibitem{HM}
C.~D.~Hacon and J.~McKernan,
On Shokurov's rational connectedness conjecture,
Duke Math. \ J. \ \textbf{138} (2007), no. 1, 119--136.

\bibitem{Ka}
Y.~Kawamata,
On the length of an extremal rational curve, Invent. Math.
\textbf{105} (1991), no. 3, 609--611.


\bibitem{KaMaMa}
Y.~Kawamata, K.~Matsuda and K.~Matsuki,
Introduction to the minimal model problem. Algebraic geometry, Sendai, 1985, 283 -- 360,
Adv. Stud. Pure Math., \textbf{10}, North-Holland, Amsterdam, 1987.

\bibitem{KeMaMc}
S.~Keel, K.~Matsuki and J.~$\text{M}^{\text c}$Kernan,
Log abundance theorem for threefolds,
Duke Math. \ J. \ \textbf{75} (1994), no. 1, 99--119.

\bibitem{KMc}
S.~Keel and J.~$\text{M}^{\text c}$Kernan, Rational curves on quasi-projective surfaces,
Mem. \ Amer. \ Math. \ Soc. \ \textbf{140} (1999), no. \textbf{669}.

\bibitem{Ko}
J.~Koll\'ar \'et al.,
Flips and abundance for algebraic threefolds,
Ast\'erisque No. \textbf{211} (1992).

\bibitem{KM}
J.~Koll\'ar and S.~Mori,
Birational geometry of algebraic varieties,
Cambridge Tracts in Math. \textbf{134},
Cambridge Univ.\ Press, 1998.

\bibitem{Lang}
S.~Lang,  Introduction to complex hyperbolic spaces, Springer-Verlag, New York, 1987.

\bibitem{Matsuki} K. Matsuki, Introduction to the Mori program,
Springer Universitex, New York, 2002.

\bibitem{MP}
M.~McQuillan and G.~Pacienza,
Remarks about bubbles,
arXiv:\textbf{1211.0203}

\bibitem{MM}
S.~Mori and S.~Mukai,
The uniruledness of the moduli space of curves of genus 11,
Algebraic geometry (Tokyo/Kyoto, 1982), 334--353,
Lecture Notes in Math., \textbf{1016}, Springer, Berlin, 1983.

\end{thebibliography}
\end{document}